\documentclass[12pt]{article}
\usepackage[margin=1in]{geometry}
\usepackage{mathtools}
\usepackage{amssymb}
\usepackage{amsthm}
\usepackage[english]{babel}
\usepackage{booktabs}
\usepackage{dynkin-diagrams}
\usepackage{titling}
\usepackage[colorlinks=true, allcolors=blue]{hyperref}

\title{A Cotangent Model for the Next-to-Minimal Orbit of $E_7$}
\author{Boming Jia}
\date{}
\hypersetup{
    pdftitle={A Cotangent Model for the Next-to-Minimal Orbit of E7},
    pdfauthor={Boming Jia},
    pdfsubject={Nilpotent orbits, cotangent bundles, and the Albert algebra},
    pdfkeywords={nilpotent orbit closures, cotangent bundles, Albert algebra,
        Bialynicki-Birula decompositions}
}

\newtheorem{theorem}{Theorem}[section]
\newtheorem{proposition}[theorem]{Proposition}
\newtheorem{lemma}[theorem]{Lemma}
\theoremstyle{definition}
\newtheorem{definition}[theorem]{Definition}

\DeclareMathOperator{\Spec}{Spec}
\DeclareMathOperator{\Ad}{Ad}
\DeclareMathOperator{\codim}{codim}
\DeclareMathOperator{\Herm}{Herm}

\begin{document}
\setlength{\droptitle}{-6em}
\maketitle
\vspace{-3.5em}

\begin{abstract}
Let \(\mathcal O_{2A_1}\) be the next-to-minimal nilpotent orbit of the
complex simple Lie algebra of type \(E_7\). Let \(\mathfrak p_7\) be the
maximal parabolic subalgebra corresponding to the simple root \(\alpha_7\),
and let \(\mathfrak u^+\) be its nilpotent radical. We prove that the
affinization \(\bigl(T^*(\mathcal O_{2A_1}\cap\mathfrak u^+)\bigr)^{\mathrm{aff}}\)
of the cotangent bundle of \(\mathcal O_{2A_1}\cap\mathfrak u^+\) is
isomorphic to the nilpotent orbit closure \(\overline{\mathcal O}_{2A_1}\)
as affine varieties.
\end{abstract}

\setlength{\abovedisplayskip}{10pt plus 2pt minus 5pt}
\setlength{\belowdisplayskip}{10pt plus 2pt minus 5pt}

\section{Introduction}

Let \(G\) be the simply connected complex simple algebraic group of type
\(E_7\). Let \(\mathfrak g\) be its Lie algebra. Fix a maximal torus
\(T\subset G\). Fix a Borel subgroup \(B\subset G\) containing \(T\).
Number the corresponding simple roots as in Bourbaki
\cite[Planche~VI]{Bourbaki}. Let \(P_7\supset B\) be the standard maximal
parabolic subgroup corresponding to the simple root \(\alpha_7\). Set
\[
    \mathfrak u^+
    \coloneqq
    \bigoplus_{\alpha\geq\alpha_7}\mathfrak g_\alpha.
\]
Then \(\mathfrak u^+\) is the nilpotent radical of
\(\operatorname{Lie}(P_7)\). The Lie algebra \(\mathfrak u^+\) is abelian
and \(27\)-dimensional. Let
\(\mathcal O_{2A_1}\subset\mathfrak g\) be the nilpotent orbit with
Bala--Carter label \(2A_1\). The orbit \(\mathcal O_{2A_1}\) is next-to-minimal.
For a variety \(Y\), define its \emph{affinization} by
\[
    Y^{\mathrm{aff}}\coloneqq\Spec\Gamma(Y,\mathcal O_Y).
\]

\begin{theorem}\label{main}
The variety \(\mathcal O_{2A_1}\cap\mathfrak u^+\) is smooth and quasi-affine.
The ring of global regular functions on its cotangent bundle is finitely
generated. There is an isomorphism of affine varieties
\[
    \bigl(T^*(\mathcal O_{2A_1}\cap\mathfrak u^+)\bigr)^{\mathrm{aff}}
    \cong\overline{\mathcal O}_{2A_1}.
\]
\end{theorem}

Our proof adapts the strategy of Fu and Liu
\cite[Lemma~4.1 and Proof of Theorem~4.2]{FuLiu}. They combine a
\(\mathbb G_m\)-action with weights \(0,1,2\) and Fu's cotangent-bundle lemma
\cite[Lemma~3.7]{FuSymplectic}. The pair
\((\mathfrak e_7,\mathcal O_{2A_1})\) is not among the pairs treated there. We
identify the fixed locus with the rank-two stratum in the Albert algebra
and prove the boundary estimate needed for this additional case. For the
boundary estimate, we use Fu's \(\mathbb Q\)-factoriality and
positive-weight contraction arguments
\cite[Proposition~4.4 and Proof of Lemma~4.3]{FuQFactorial}.

\vspace{0.8em}
\textbf{Acknowledgments.}
The author was supported by NSFC Grant No.~12225108 and the Shuimu
Scholar Program at Tsinghua University. The author used ChatGPT
extensively during the preparation of this work. He would like to
thank Baohua Fu for useful discussions. He
would also like to thank the model GPT-5.6 Sol for generating an earlier
draft of this paper and collaborating with the author throughout the
revision process.

\section{The rank-two locus in the Albert algebra}

All algebraic background on composition algebras and Albert algebras used in
this section is taken from
\cite[Sections~1.2--1.4, 1.6, 1.8, 5.1, and~5.2]{SV}. We recall only the
definitions and normalizations needed below.

A \emph{unital composition algebra} over \(\mathbb C\) is a finite-dimensional
\(\mathbb C\)-algebra with identity, not necessarily associative, equipped with a
quadratic form \(n\) whose normalized polar form
\[
    \langle z,w\rangle
    \coloneqq\frac{n(z+w)-n(z)-n(w)}2
\]
is nondegenerate and for which
\[
    n(zw)=n(z)n(w).
\]
Its identity element is denoted by \(1\). Multiplicativity and
nondegeneracy imply \(n(1)=1\). An \emph{octonion algebra} is an
eight-dimensional unital composition algebra. An octonion algebra is
\emph{split} if its norm is isotropic, that is, if \(n(z)=0\) for some
nonzero \(z\). There is, up to isomorphism, a unique split octonion algebra
over \(\mathbb C\)
\cite[Theorem~1.8.1]{SV}. We denote it by \(\mathbb O_{\mathbb C}\).
The algebra \(\mathbb O_{\mathbb C}\) is alternative but not associative
\cite[Lemma~1.4.2 and Theorem~1.6.2]{SV}. We therefore parenthesize products
of three octonions.

Define the \emph{octonionic trace} and \emph{standard conjugation} by
\[
    \operatorname{tr}_{\mathbb O}(z)
    \coloneqq2\langle z,1\rangle,
    \qquad
    \bar z\coloneqq\operatorname{tr}_{\mathbb O}(z)1-z.
\]
By \cite[Lemma~1.3.1]{SV}, the trace and conjugation satisfy
\[
    \overline{zw}=\bar w\,\bar z,
    \qquad
    \overline{\bar z}=z,
    \qquad
    z\bar z=\bar z z=n(z)1,
    \qquad
    z+\bar z=\operatorname{tr}_{\mathbb O}(z)1.
\]
The definition and the last equality show that the fixed-point space of
conjugation is \(\mathbb C\cdot1\).

\begin{definition}
Let
\[
    J\coloneqq\Herm_3(\mathbb O_{\mathbb C})
    =\{x\in\operatorname{Mat}_3(\mathbb O_{\mathbb C})\mid x^*=x\},
\]
where \(x^*=(\bar x_{ji})\) is the conjugate transpose. The space \(J\) is
\(27\)-dimensional. Every element of \(J\) has a unique expression
\[
    x=
    \begin{pmatrix}
        a_1&z_3&\bar z_2\\
        \bar z_3&a_2&z_1\\
        z_2&\bar z_1&a_3
    \end{pmatrix},
    \qquad
    a_i\in\mathbb C,\quad z_i\in\mathbb O_{\mathbb C}.
\]
The usual row-by-column formula defines the binary products \(xy\) and
\(yx\). Each summand involves only a product of two octonions. Since
\((xy)^*=y^*x^*\), the element \((xy+yx)/2\) lies in \(J\). Hence
\[
    x\circ y\coloneqq\frac{xy+yx}{2}
\]
makes \(J\) a commutative algebra with identity \(I_3\). The product
\(\circ\) satisfies the Jordan identity
\[
    (x^2\circ y)\circ x=x^2\circ(y\circ x),
    \qquad x^2\coloneqq x\circ x.
\]
Thus \(J\) is the \emph{split Albert algebra}
\cite[Equations~(5.1)--(5.2) and Proposition~5.1.6]{SV}. Its trace is
\[
    \operatorname{tr}(x)\coloneqq a_1+a_2+a_3.
\]
For \(x=(a_i,z_i)\) and
\(y=(b_i,w_i)\), the \emph{trace form} is
\[
    T(x,y)\coloneqq\operatorname{tr}(x\circ y)
    =\sum_{i=1}^3a_ib_i
       +2\sum_{i=1}^3\langle z_i,w_i\rangle.
\]
The displayed formula and the nondegeneracy of \(\langle\ ,\ \rangle\)
show that \(T\) is nondegenerate. The \emph{cubic norm} on \(J\), also
called the \emph{determinant}, is
\[
    N(x)\coloneqq a_1a_2a_3-a_1n(z_1)-a_2n(z_2)-a_3n(z_3)
           +2\langle z_1z_2,\bar z_3\rangle.
\]
The displayed formula is \cite[Equation~(5.11)]{SV} specialized to the
matrix model above and written using our normalized polar form. For
\(x,y\in J\), set
\[
    dN_x(y)
    \coloneqq
    \left.\frac{d}{dt}\right|_{t=0}N(x+ty).
\]
Since \(T\) is nondegenerate, the linear functional \(dN_x\) has a unique
representative \(x^\#\in J\) such that
\begin{equation}\label{sharp-definition}
    T(x^\#,y)=dN_x(y)
    \qquad\text{for every }y\in J.
\end{equation}
We call \(x^\#\) the \emph{adjoint} of \(x\).
Since \(N\) is cubic, the map \(x\mapsto x^\#\) is a homogeneous quadratic
polynomial map.
By Euler's identity,
\[
    T(x^\#,x)=dN_x(x)=3N(x).
\]
Thus \(x^\#=0\) implies \(N(x)=0\).
For a diagonal element, the formulas for \(N\) and the adjoint become
\[
    N(\operatorname{diag}(a_1,a_2,a_3))=a_1a_2a_3,
    \qquad
    \operatorname{diag}(a_1,a_2,a_3)^\#
    =\operatorname{diag}(a_2a_3,a_1a_3,a_1a_2).
\]
The implication above shows that every element of \(J\) lies in exactly one
of the following four cases.
We define its \emph{Jordan rank} by
\[
    \operatorname{rk}_J(x)
    \coloneqq
    \begin{cases}
        0,&x=0,\\
        1,&x\ne0\text{ and }x^\#=0,\\
        2,&x^\#\ne0\text{ and }N(x)=0,\\
        3,&N(x)\ne0.
    \end{cases}.
\]
For \(0\leq r\leq3\), set
\[
    J_r\coloneqq\{x\in J\mid\operatorname{rk}_J(x)=r\},
    \qquad
    J_{\leq r}\coloneqq\bigcup_{s=0}^rJ_s.
\]
In particular, the two closed rank loci and the rank-two stratum are
\[
    \begin{aligned}
        J_{\leq1}&=\{x\in J\mid x^\#=0\},\\
        J_{\leq2}&=\{x\in J\mid N(x)=0\},\\
        J_2&=\{x\in J\mid N(x)=0,\ x^\#\ne0\}.
    \end{aligned}.
\]
Thus a diagonal element has Jordan rank equal to the number of its nonzero
diagonal entries. We use no other notion of rank below.
\end{definition}

\begin{lemma}\label{Albert-irreducible}
The cubic polynomial \(N\) is irreducible.
\end{lemma}

\begin{proof}
Let \(V\subset J\) be the subspace defined by \(z_i\in\mathbb C\cdot1\) for
every \(i\). Then \(V\cong\operatorname{Sym}_3(\mathbb C)\). Under the
identification \(V\cong\operatorname{Sym}_3(\mathbb C)\), the restriction of
\(N\) is the ordinary determinant.

We first prove that the determinant on \(\operatorname{Sym}_3(\mathbb C)\) is
irreducible. The zero set of the determinant consists of the symmetric
matrices of rank at most two. Every symmetric bilinear form over \(\mathbb C\)
is congruent to \(\operatorname{diag}(I_r,0)\). Hence every symmetric matrix of rank at
most two has the form \(AA^{\mathsf t}\) for some
\(A\in\operatorname{Mat}_{3\times2}(\mathbb C)\). Conversely, every matrix
\(AA^{\mathsf t}\) has rank at most two. Thus the morphism
\[
    \operatorname{Mat}_{3\times2}(\mathbb C)\longrightarrow
    \operatorname{Sym}_3(\mathbb C),
    \qquad A\longmapsto AA^{\mathsf t},
\]
is surjective onto the determinant hypersurface. Since its source is
irreducible, that hypersurface is irreducible. At
\(\operatorname{diag}(1,1,0)\), the directional derivative along
\(\operatorname{diag}(0,0,1)\) is \(1\). Thus the determinant is not a
nontrivial power of the
irreducible equation defining the determinant hypersurface. Therefore the
determinant polynomial is irreducible.

Suppose that \(N=FH\) is a nontrivial factorization. Since \(N\) is
homogeneous, we may take \(F\) and \(H\) to be homogeneous of positive degree.
Since \(N|_V\ne0\), neither \(F|_V\) nor \(H|_V\) is zero. Both
restrictions remain homogeneous of positive degree because \(V\) is a
linear subspace through the origin. Thus \(N|_V\) is reducible, a
contradiction.
\end{proof}

\begin{proposition}\label{rank-two-geometry}
The stratum \(J_2\) is a smooth irreducible quasi-affine variety of
dimension \(26\). The hypersurface \(J_{\leq2}\) is normal and irreducible.
For \(0\leq r\leq3\),
\[
    \overline{J_r}=\bigcup_{s=0}^rJ_s,
\]
where the closures are taken in \(J\).
\end{proposition}

\begin{proof}
By Lemma~\ref{Albert-irreducible}, the hypersurface \(J_{\leq2}\) is
irreducible and has dimension \(26\). The diagonal elements
\(\operatorname{diag}(1,1,0)\) and \(\operatorname{diag}(1,1,1)\) show
that \(J_2\) and \(J_3\) are nonempty. The singular locus of
\(J_{\leq2}\) is \(\{x\in J\mid dN_x=0\}\). Since \(T\) is nondegenerate,
\eqref{sharp-definition} shows that
\[
    \{x\in J\mid dN_x=0\}=J_{\leq1}.
\]

By \cite[Sections~3.5 and~4.1]{LM}, the projectivization of
\(J_{\leq1}\) is the Cayley plane. By \cite[Proposition~2.1]{LM}, the
Cayley plane has dimension \(16\), so \(\dim J_{\leq1}=17\). Thus
\[
    J_2=J_{\leq2}\setminus J_{\leq1}
\]
is the smooth locus of \(J_{\leq2}\). Hence \(J_2\) is smooth, irreducible,
quasi-affine, and \(26\)-dimensional.

The hypersurface \(J_{\leq2}\) is Cohen--Macaulay and therefore
satisfies \(S_2\). The singular locus of \(J_{\leq2}\) has codimension
nine in \(J_{\leq2}\). Hence \(J_{\leq2}\) satisfies \(R_1\). By Serre's
criterion, \(J_{\leq2}\) is normal. Since \(J_2\) is dense in
\(J_{\leq2}\), its closure in \(J\) is \(J_{\leq2}\). The cone
\(J_{\leq1}\) is irreducible because its projectivization is the Cayley
plane. Hence \(\overline{J_1}=J_{\leq1}\). Finally, \(J_3\) is a nonempty
open subset of \(J\), so it is dense. Since \(J_0=\{0\}\), it is closed.
This proves all four closure relations.
\end{proof}

\section{The attracting set}

We adapt the attracting-set argument of Fu and Liu
\cite[Section~4.2, Lemma~4.1, and Proof of Theorem~4.2]{FuLiu} to the
pair \((\mathfrak g,\mathcal O_{2A_1})\). Their proof applies the \(\mathbb G_m\)-action
below and Fu's cotangent-bundle argument
\cite[Proof of Lemma~3.7]{FuSymplectic} to the inverse image of the fixed
locus. The two ingredients specific to the present case are the
description of \(\mathcal O_{2A_1}\cap\mathfrak u^+\) and a rank bound for the
projection onto \(\mathfrak u^+\). We prove both facts below.
Section~\ref{boundary-section} proves the global boundary
statement needed to identify the affinization of the cotangent bundle
with the orbit closure.

Let \(L\subset P_7\) be the standard Levi subgroup containing \(T\). We
identify the derived subgroup of \(L\) with the simply connected group of
type \(E_6\). If necessary, we compose the identification with the
nontrivial diagram automorphism so that the \(E_6\)-module \(\mathfrak u^+\)
is the Albert module \(J\). Let \(P_7^-\) be the opposite parabolic subgroup with
Levi subgroup \(L\). Put
\[
    \mathfrak l\coloneqq\operatorname{Lie}(L),
    \qquad
    \mathfrak u^-
    \coloneqq
    \operatorname{nilrad}(\operatorname{Lie}(P_7^-)).
\]
The Lie algebra \(\mathfrak g\) has the three-step grading
\begin{equation}\label{grading}
    \mathfrak g
    =
    \mathfrak u^+\oplus\mathfrak l\oplus\mathfrak u^-.
\end{equation}
The three summands have degrees \(1,0,-1\), respectively. Both
\(\mathfrak u^+\) and \(\mathfrak u^-\) are abelian and
\(27\)-dimensional. We use the following Bourbaki numbering.
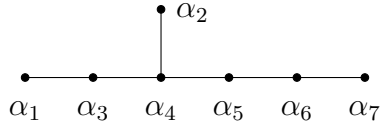
\[
    \dynkin[edge length=.9cm,text style/.default={black,font=\normalsize},labels={\alpha_1,\alpha_2,\alpha_3,\alpha_4,\alpha_5,\alpha_6,\alpha_7}]{E}{7}
\]

For every root \(\beta\), fix a root vector
\[
    0\ne e_\beta\in\mathfrak g_\beta.
\]
Fix an \(E_6\)-equivariant linear isomorphism
\[
    \mathfrak u^+\xrightarrow{\sim}J.
\]
Via the chosen isomorphism, we regard the cubic form \(N\), the adjoint
\(x\mapsto x^\#\), and the rank strata as structures on \(\mathfrak u^+\).
By \cite[Section~1.3]{LM}, the group \(E_6\) preserves \(N\). The center
of \(L\) acts on \(J\) by scalar multiplication.

\begin{proposition}\label{intersection}
Under the chosen identification of \(\mathfrak u^+\) with \(J\),
\[
    \mathcal O_{2A_1}\cap\mathfrak u^+
    =J_2.
\]
In particular, \(\mathcal O_{2A_1}\cap\mathfrak u^+\) is smooth and quasi-affine
of dimension \(26\).
\end{proposition}

\begin{proof}
Let \(U^+\subset P_7\) be the unipotent radical. Since \(U^+\) is
abelian, the exponential map is an \(L\)-equivariant isomorphism of
algebraic groups \(\mathfrak u^+\xrightarrow{\sim}U^+\). By
\cite[Propositions~2.8 and~2.13]{RRS}, every \(L\)-orbit in
\(\mathfrak u^+\) contains a sum of root vectors for pairwise
orthogonal degree-one roots. Two such sums lie in the same \(L\)-orbit
if and only if they have the same number of summands. By
\cite[Proposition~4.1]{LM}, the images of the three nonzero rank strata
\(J_1,J_2,J_3\) in \(\mathbb P(J)\) are precisely the three \(E_6\)-orbits.
Since the
center of \(L\) acts by scalars, the four
\(L\)-orbits in \(J\) are \(J_0,J_1,J_2,J_3\).
By \cite[Proposition~2.15(a)]{RRS}, the closure
of the \(L\)-orbit represented by \(r\) summands is the union of the
\(L\)-orbits represented by at most \(r\) summands. By
Proposition~\ref{rank-two-geometry}, the rank loci satisfy
\(\overline{J_r}=\bigcup_{s\leq r}J_s\). Hence the \(L\)-orbit represented
by \(r\) summands is \(J_r\).
The stratum \(J_2\) contains an element
\[
    e_2\coloneqq e_{\beta_1}+e_{\beta_2},
\]
where \(\beta_1\) and \(\beta_2\) are orthogonal degree-one roots.
Choose root \(\mathfrak{sl}_2\)-triples
\[
    (e_{\beta_i},\beta_i^\vee,f_{\beta_i})
    \qquad(i=1,2).
\]
The two triples commute. The Weyl group of \(G\) is transitive on ordered
pairs of orthogonal roots. Indeed, choose \(w_1\in W(G,T)\) such that
\(w_1(\beta_1)=\alpha_1\). The roots orthogonal to \(\alpha_1\) form a
root system of type \(D_6\). Its Weyl group is transitive on its roots.
Hence there is an element \(w_2\) in this Weyl group such that
\[
    w_2w_1(\beta_2)=\alpha_2.
\]
The element \(w_2\) fixes \(\alpha_1\). Choose a representative in
\(N_G(T)\) of \(w_2w_1\). It sends \(e_2\) to
\[
    c_1e_{\alpha_1}+c_2e_{\alpha_2}
    \qquad(c_1,c_2\in\mathbb C^\times).
\]
The element \(c_1e_{\alpha_1}+c_2e_{\alpha_2}\) is a regular nilpotent
element in the derived Lie algebra of the standard Levi subgroup corresponding to
\(\{\alpha_1,\alpha_2\}\). The element is therefore distinguished in that
Lie algebra. By the Bala--Carter classification, the \(G\)-orbit of \(e_2\)
has label \(2A_1\). The stratum \(J_1\) contains a root vector.
Since \(J_1\) is a single \(L\)-orbit, every point of \(J_1\) is
\(G\)-conjugate to a root vector. Hence \(J_1\) is contained in the
minimal nilpotent orbit \(\mathcal O_{A_1}\).

The unipotent radical of \(P_7\) acts trivially on \(\mathfrak u^+\).
Hence \(J_3\) is the open \(P_7\)-orbit in \(\mathfrak u^+\). Every point
of \(J_3\) is therefore a Richardson element for \(P_7\). The \(G\)-orbit
of every point of \(J_3\) has dimension
\[
    2\dim\mathfrak u^+=54
\]
by \cite[Theorem~7.1.1]{CM}. On the other hand,
\(\dim\mathcal O_{2A_1}=52\) by \cite[Section~8.4]{CM}. Therefore
\[
    J_3\cap\mathcal O_{2A_1}=\varnothing.
\]

The intersection \(\mathcal O_{2A_1}\cap\mathfrak u^+\) is \(L\)-stable.
It contains \(e_2\), whose \(L\)-orbit is \(J_2\). Thus
\[
    J_2\subset\mathcal O_{2A_1}\cap\mathfrak u^+.
\]
The remaining strata satisfy \(J_0=\{0\}\),
\(J_1\subset\mathcal O_{A_1}\), and
\(J_3\cap\mathcal O_{2A_1}=\varnothing\). Hence
\[
    \mathcal O_{2A_1}\cap\mathfrak u^+=J_2.
\]
By Proposition~\ref{rank-two-geometry}, \(J_2\) is smooth,
quasi-affine, and \(26\)-dimensional.
\end{proof}

Set
\[
    \mathcal O\coloneqq\mathcal O_{2A_1}.
\]
Let
\[
    \omega_{\mathrm{KK}}\in\Gamma(\mathcal O,\Omega^2_{\mathcal O})
\]
be the Kostant--Kirillov symplectic form. Following
\cite[Section~4.2 and Lemma~4.1]{FuLiu}, consider the
\(\mathbb G_m\)-action \(\phi\) on \(\mathfrak g\) with weights \(0,1,2\)
on \(\mathfrak u^+,\mathfrak l,\mathfrak u^-\), respectively:
\[
    \phi_t(x_++x_0+x_-)
    \coloneqq
    x_++t x_0+t^2x_-.
\]
Equivalently,
\[
    \phi_t=t\Ad(\bar\eta(t)^{-1}),
\]
where \(\bar\eta:\mathbb G_m\to G^{\mathrm{ad}}\) is the cocharacter whose
adjoint action has weights \(1,0,-1\) on the three summands of
\eqref{grading}. By \cite[Lemma~4.1]{FuLiu}, every nilpotent orbit is
stable under \(\phi\). Hence \(\mathcal O\) and \(\overline{\mathcal O}\) are
\(\phi\)-stable.

Let
\[
    \pi_+:\mathfrak g\longrightarrow\mathfrak u^+
\]
be the linear projection along
\(\mathfrak l\oplus\mathfrak u^-\). The projection \(\pi_+\) is
\(\mathbb G_m\)-equivariant for the trivial action on \(\mathfrak u^+\). For
\(y\in\mathfrak g\),
\[
    \pi_+(y)=\lim_{t\to0}\phi_t(y).
\]

\begin{lemma}\label{fixed-data}
The action \(\phi\) and the projection \(\pi_+\) satisfy
\[
    \mathcal O^{\mathbb G_m}=\mathcal O\cap\mathfrak u^+=J_2,
    \qquad
    \pi_+(\mathcal O)\subset J_{\leq2}.
\]
\end{lemma}

\begin{proof}
The fixed subspace of \(\phi\) on \(\mathfrak g\) is \(\mathfrak u^+\).
By Proposition~\ref{intersection}, we have
\[
    \mathcal O^{\mathbb G_m}=\mathcal O\cap\mathfrak u^+=J_2.
\]
The closed subset \(\overline{\mathcal O}\cap\mathfrak u^+\) contains \(J_2\).
By Proposition~\ref{rank-two-geometry}, we have
\[
    J_{\leq2}=\overline{J_2}
    \subset\overline{\mathcal O}\cap\mathfrak u^+.
\]
Suppose that \(\overline{\mathcal O}\cap\mathfrak u^+\) contains a rank-three
point. The intersection \(\overline{\mathcal O}\cap\mathfrak u^+\) is
\(L\)-stable. Hence
\[
    J_3\subset\overline{\mathcal O}\cap\mathfrak u^+.
\]
The unipotent radical of \(P_7\) acts trivially on \(\mathfrak u^+\).
Hence \(J_3\) is the open \(P_7\)-orbit in \(\mathfrak u^+\). Every point
of \(J_3\) is a Richardson element for \(P_7\). The \(G\)-orbit of every point of \(J_3\)
has dimension \(54\) by \cite[Theorem~7.1.1]{CM}. The \(G\)-stability of
\(\overline{\mathcal O}\) would then imply
\[
    \dim\overline{\mathcal O}\geq54.
\]
The inequality contradicts \(\dim\overline{\mathcal O}=52\)
\cite[Section~8.4]{CM}. Hence
\[
    \overline{\mathcal O}\cap\mathfrak u^+=J_{\leq2}.
\]
For \(y\in\mathcal O\), we have \(\phi_t(y)\in\mathcal O\) for \(t\ne0\). The
projection formula above therefore implies
\[
    \pi_+(y)\in\overline{\mathcal O}\cap\mathfrak u^+=J_{\leq2}.
    \qedhere
\]
\end{proof}

\begin{lemma}\label{symplectic-weights}
The action \(\phi\) satisfies
\[
    \phi_t^*\omega_{\mathrm{KK}}=t\omega_{\mathrm{KK}}.
\]
The fixed locus \(\mathcal O\cap\mathfrak u^+\) is Lagrangian in \(\mathcal O\), and
\(\mathbb G_m\) acts with weight one on its normal bundle.
\end{lemma}

\begin{proof}
The automorphism \(\Ad(\bar\eta(t)^{-1})\) preserves
\(\omega_{\mathrm{KK}}\). Scalar multiplication by \(t\) multiplies the
form by \(t\). Hence
\(\phi_t^*\omega_{\mathrm{KK}}=t\omega_{\mathrm{KK}}\).

Fix \(x\in\mathcal O_{2A_1}\cap\mathfrak u^+\) and write
\[
    T_x\mathcal O=V_0\oplus V_1\oplus V_2,
\]
where \(V_i\) is the weight-\(i\) subspace. Since \(x\) is fixed by
\(\phi\), the differential of \(\phi\) preserves \(T_x\mathcal O\). Under the
canonical identification \(T_x\mathfrak g=\mathfrak g\), the differential
has weights \(0,1,2\). Hence \(T_x\mathcal O\) has no other weights. For
\(v\in V_i\) and \(w\in V_j\), we have
\[
    t^{i+j}\omega_{\mathrm{KK},x}(v,w)
    = (\phi_t^*\omega_{\mathrm{KK}})_x(v,w)
    = t\omega_{\mathrm{KK},x}(v,w).
\]
Thus \(\omega_{\mathrm{KK},x}(V_i,V_j)=0\) unless \(i+j=1\).
The nondegeneracy of \(\omega_{\mathrm{KK},x}\) implies that \(V_2=0\)
and that \(V_0\) and \(V_1\) are perfectly paired.

The fixed locus of a torus acting on a smooth variety is smooth. Its tangent
space at a fixed point is the zero-weight subspace. By
Lemma~\ref{fixed-data}, the fixed locus of \(\phi\) on \(\mathcal O\) is
\(\mathcal O_{2A_1}\cap\mathfrak u^+\). Hence
\[
    V_0=T_x(\mathcal O_{2A_1}\cap\mathfrak u^+).
\]
The symplectic form vanishes on \(V_0\). Since \(V_0\) and \(V_1\) are
perfectly paired,
\[
    \dim V_0=\dim V_1=\frac12\dim\mathcal O.
\]
Thus \(\mathcal O_{2A_1}\cap\mathfrak u^+\) is Lagrangian. The normal space
\(T_x\mathcal O/V_0\) is the weight-one space \(V_1\). Since \(x\) was arbitrary,
\(\mathbb G_m\) acts with weight one on the normal bundle.
\end{proof}

Let \(\mathcal A\) be the \emph{attracting set} of the fixed locus
\(\mathcal O_{2A_1}\cap\mathfrak u^+\):
\[
    \mathcal A
    \coloneqq
    \bigl(\left.\pi_+\right|_{\mathcal O}\bigr)^{-1}
    (\mathcal O_{2A_1}\cap\mathfrak u^+).
\]
Equivalently, \(\mathcal A\) consists of those \(y\in\mathcal O\) for which
\[
    \lim_{t\to0}\phi_t(y)\in\mathcal O_{2A_1}\cap\mathfrak u^+.
\]
By Lemma~\ref{fixed-data}, the image of \(\pi_+|_{\mathcal O}\) is contained in
\(J_{\leq2}\). By Proposition~\ref{intersection}, the fixed locus is
\(J_{\leq2}\setminus J_{\leq1}\). Hence
\[
    \mathcal A
    =
    \mathcal O\setminus
    \bigl(\left.\pi_+\right|_{\mathcal O}\bigr)^{-1}(J_{\leq1}).
\]
The equality above shows that \(\mathcal A\) is open in \(\mathcal O\). The attracting
set contains the fixed locus and is therefore nonempty. Since \(\mathcal O\) is
irreducible, \(\mathcal A\) is dense in \(\mathcal O\).
Let \(\mathbb G_m\) act on \(T^*(\mathcal O_{2A_1}\cap\mathfrak u^+)\) by fiberwise
scalar multiplication:
\[
    t\cdot(x,\xi)=(x,t\xi).
\]
Proposition~\ref{cotangent-open} is the quasi-affine attracting-set
construction of Fu and Liu \cite[Proof of Theorem~4.2]{FuLiu}. Their proof
adapted Fu's argument \cite[Proof of Lemma~3.7]{FuSymplectic}.

\begin{proposition}\label{cotangent-open}
The morphism
\[
    \left.\pi_+\right|_{\mathcal A}:
    \mathcal A\longrightarrow\mathcal O_{2A_1}\cap\mathfrak u^+
\]
is a vector bundle. There is a \(\mathbb G_m\)-equivariant isomorphism over
\(\mathcal O_{2A_1}\cap\mathfrak u^+\):
\[
    \mathcal A\xrightarrow{\sim}
    T^*(\mathcal O_{2A_1}\cap\mathfrak u^+).
\]
\end{proposition}

\begin{proof}
Let
\[
    q:\mathcal A\longrightarrow\mathcal O_{2A_1}\cap\mathfrak u^+,
    \qquad
    q(y)\coloneqq\pi_+(y).
\]
By the Bia{\l}ynicki-Birula decomposition, the attraction morphism \(q\)
is a Zariski locally trivial affine-space bundle over \(J_2\). By
Lemma~\ref{symplectic-weights}, every positive normal weight is one. Fu's
argument \cite[Proof of Lemma~3.7]{FuSymplectic}, as applied by
Fu and Liu \cite[Proof of Theorem~4.2]{FuLiu}, shows that \(q\) is a
\(\mathbb G_m\)-equivariant vector bundle of rank \(26\). The action of
\(\mathbb G_m\) on its fibers is scalar multiplication. The zero section of \(q\)
is the inclusion
\[
    \mathcal O_{2A_1}\cap\mathfrak u^+
    \longrightarrow\mathcal A.
\]
Let \(\mathcal N_{\mathcal A}\) and \(\mathcal N_{\mathcal O}\) be the normal
bundles of \(J_2\) in \(\mathcal A\) and \(\mathcal O\), respectively. The
vector-bundle structure of \(q\) induces a \(\mathbb G_m\)-equivariant isomorphism
from \(\mathcal A\) to \(\mathcal N_{\mathcal A}\). The open immersion
\(\mathcal A\hookrightarrow\mathcal O\) induces a \(\mathbb G_m\)-equivariant
isomorphism from \(\mathcal N_{\mathcal A}\) to \(\mathcal N_{\mathcal O}\). By
Lemma~\ref{symplectic-weights}, the Kostant--Kirillov form induces a
\(\mathbb G_m\)-equivariant isomorphism from \(\mathcal N_{\mathcal O}\) to the
cotangent bundle of \(\mathcal O_{2A_1}\cap\mathfrak u^+\).
Thus
\[
    \mathcal A
    \cong\mathcal N_{\mathcal A}
    \cong\mathcal N_{\mathcal O}
    \cong T^*(\mathcal O_{2A_1}\cap\mathfrak u^+).
    \qedhere
\]
\end{proof}

\section{The boundary and the affinization}\label{boundary-section}

We use Fu's classification of exceptional nilpotent orbit closures whose
normalizations are \(\mathbb Q\)-factorial
\cite[Proposition~4.4]{FuQFactorial}. We also use his positive-weight
argument for the Picard group
\cite[Proof of Lemma~4.3]{FuQFactorial}. The proof of
Theorem~\ref{boundary-in-closure} applies these two results to
\(\overline{\mathcal O}_{2A_1}\).

\begin{lemma}\label{normality}
The orbit closure \(\overline{\mathcal O}_{2A_1}\) is normal.
\end{lemma}

\begin{proof}
Let
\[
    e\coloneqq e_{\beta_1}+e_{\beta_2},
    \qquad
    h\coloneqq\beta_1^\vee+\beta_2^\vee,
\]
where \(e\) is the representative chosen in
Proposition~\ref{intersection}. The two root triples in the proof of that
proposition commute. Hence
\((e,h,f_{\beta_1}+f_{\beta_2})\) is an \(\mathfrak{sl}_2\)-triple. The
root system of \(G\) is simply laced. Thus, if a root \(\delta\) is distinct
from \(\pm\beta_1\) and \(\pm\beta_2\), then
each Cartan integer \(\langle\delta,\beta_i^\vee\rangle\) belongs to
\(\{-1,0,1\}\). If \(\delta=\pm\beta_i\), then its \(h\)-weight is
\(\pm2\) because \(\beta_1\) and \(\beta_2\) are orthogonal. The Cartan
subalgebra has \(h\)-weight zero. Hence every \(h\)-weight on
\(\mathfrak g\) lies between \(-2\) and \(2\). Since \([h,e]=2e\),
the orbit \(\mathcal O_{2A_1}\) has height \(2\).
By \cite[Proposition~4.1.1]{BenderPerrin}, the height-two orbit closure
\(\overline{\mathcal O}_{2A_1}\) is normal.
\end{proof}

\begin{theorem}\label{boundary-in-closure}
The complement of \(\mathcal A\) in \(\overline{\mathcal O}_{2A_1}\) has
codimension at least two. Hence restriction induces an isomorphism
\[
    \mathbb C[\overline{\mathcal O}_{2A_1}]\xrightarrow{\sim}
    \Gamma(\mathcal A,\mathcal O_{\mathcal A}).
\]
\end{theorem}

\begin{proof}
Suppose that an irreducible component \(D\) of
\(\overline{\mathcal O}_{2A_1}\setminus\mathcal A\) has codimension one.
By Lemma~\ref{normality}, the variety \(\overline{\mathcal O}_{2A_1}\) is
normal and therefore coincides with its normalization. By Fu's
classification \cite[Proposition~4.4]{FuQFactorial}, it is
\(\mathbb Q\)-factorial. Hence \(mD\) is Cartier for some
\(m>0\). Scalar dilation is a positive-weight action on
\(\overline{\mathcal O}_{2A_1}\) with the origin as its unique fixed point. By
Fu's positive-weight argument
\cite[Proof of Lemma~4.3]{FuQFactorial}, we have
\[
    \operatorname{Pic}(\overline{\mathcal O}_{2A_1})=0.
\]
Thus the Cartier divisor \(mD\) has zero class in the Picard group. The
divisor \(mD\) is therefore principal. Choose
\[
    f\in\mathbb C(\overline{\mathcal O}_{2A_1})^\times,
    \qquad
    \operatorname{div}(f)=mD.
\]
Since \(\operatorname{div}(f)=mD\) is effective and
\(\overline{\mathcal O}_{2A_1}\) is normal, the rational function \(f\) is
regular. Its divisor is supported on the complement of \(\mathcal A\),
so \(f|_{\mathcal A}\) is invertible.

We claim that every invertible regular function on \(\mathcal A\) is
constant. By Propositions~\ref{intersection} and~\ref{cotangent-open},
the variety \(\mathcal A\) is isomorphic to \(T^*J_2\). Let
\[
    p:T^*J_2\longrightarrow J_2,
    \qquad
    s_0:J_2\longrightarrow T^*J_2
\]
be the projection and the zero section, respectively. If \(u\) is an
invertible regular function on \(T^*J_2\),
then its restriction to each fiber of \(p\) is constant. Hence
\(u=p^*(s_0^*u)\).

We must show that \(s_0^*u\) is constant. By
Proposition~\ref{rank-two-geometry},
\[
    \codim_{J_{\leq2}}(J_{\leq2}\setminus J_2)
    =\codim_{J_{\leq2}}J_{\leq1}=26-17=9.
\]
By the Hartogs Lemma for the normal variety \(J_{\leq2}\),
\[
    \Gamma(J_2,\mathcal O_{J_2})=\mathbb C[J_{\leq2}].
\]
The ring \(\mathbb C[J_{\leq2}]\) is a nonnegatively graded domain with
degree-zero part \(\mathbb C\). Let \(a,b\in\mathbb C[J_{\leq2}]\) satisfy \(ab=1\).
The product of their highest-degree components is nonzero because the
ring is a domain. Hence both highest degrees are zero, and
\(a,b\in\mathbb C^\times\). Taking \(a=s_0^*u\) and \(b=(s_0^*u)^{-1}\), we
conclude that \(s_0^*u\) is constant. Thus \(u\) is constant.

The invertible function \(f|_{\mathcal A}\) is therefore constant. The
density of \(\mathcal A\) in \(\overline{\mathcal O}_{2A_1}\) implies that \(f\)
is constant. This contradicts \(\operatorname{div}(f)=mD\). Therefore the
complement of \(\mathcal A\) has codimension at least two.
By the Hartogs Lemma for normal varieties, restriction induces the
stated isomorphism.
\end{proof}

\begin{proof}[Proof of Theorem~\ref{main}]
By Proposition~\ref{intersection}, the variety
\(\mathcal O_{2A_1}\cap\mathfrak u^+\) is smooth and quasi-affine. By
Proposition~\ref{cotangent-open} and
Theorem~\ref{boundary-in-closure}, we have isomorphisms of
\(\mathbb C\)-algebras
\[
    \Gamma\bigl(T^*(\mathcal O_{2A_1}\cap\mathfrak u^+),
        \mathcal O_{T^*(\mathcal O_{2A_1}\cap\mathfrak u^+)}\bigr)
    \xrightarrow{\sim}
    \Gamma(\mathcal A,\mathcal O_{\mathcal A})
    \xleftarrow{\sim}
    \mathbb C[\overline{\mathcal O}_{2A_1}].
\]
Hence the ring of global regular functions on the cotangent bundle is
finitely generated. After taking spectra, we obtain an isomorphism
\[
    \bigl(T^*(\mathcal O_{2A_1}\cap\mathfrak u^+)\bigr)^{\mathrm{aff}}
    \cong\overline{\mathcal O}_{2A_1}
\]
of affine varieties.
\end{proof}

\section{The classification}

A nilpotent orbit closure \(\overline{\mathcal O}\) admits a \emph{cotangent
model} if there is a smooth quasi-affine variety \(X\) such that
\[
    \overline{\mathcal O}\cong(T^*X)^{\mathrm{aff}}
\]
as affine varieties.

We follow Fu and Liu
\cite[Section~4.2, Table~1, and Theorems~1.4 and~4.2]{FuLiu}. For a
parabolic subgroup \(P\), define
\[
    \mathfrak n_P
    \coloneqq
    \operatorname{nilrad}(\operatorname{Lie}(P)).
\]
For every pair in their Table~1, Fu and Liu set
\[
    X\coloneqq\mathcal O\cap\mathfrak n_P
\]
and proved that
\[
    (T^*X)^{\mathrm{aff}}\cong\overline{\mathcal O}
\]
as affine varieties.

\begin{theorem}
Let \(\mathcal O\) be a nonzero nilpotent orbit in a complex simple Lie algebra
\(\mathfrak g\). Then \(\overline{\mathcal O}\) admits a cotangent model if and
only if \((\mathfrak g,\mathcal O)\) occurs in the following table.
\begin{center}
\begingroup
\normalfont\small
\renewcommand{\arraystretch}{1.08}
\begin{tabular*}{\linewidth}{@{\extracolsep{\fill}}c l l l@{}}
\toprule
\textbf{Type}&\textbf{Orbit}&\textbf{Conditions}&
    \textbf{Parabolics} \(P\)\\
\midrule
\(A_{n-1}\)
    &\(\mathcal O_{(2^r,1^{n-2r})}\)
    &\(n\geq4,\quad1\leq r\leq\lfloor(n-2)/2\rfloor\)
    &\(P_k,\quad r<k<n-r\)\\
\addlinespace[2pt]
\(B_n\)
    &\(\mathcal O_{(2^2,1^{2n-3})}\)
    &\(n\geq3\)
    &\(P_1\)\\
\addlinespace[2pt]
\(C_n\)
    &\(\mathcal O_{(2^r,1^{2n-2r})}\)
    &\(n\geq2,\quad1\leq r\leq n-1\)
    &\(P_n\)\\
\addlinespace[2pt]
\(D_n\)
    &\(\mathcal O_{(2^{2r},1^{2n-4r})}\)
    &\(n\geq4,\quad1\leq r\leq\lfloor(n-2)/2\rfloor\)
    &\(P_{n-1},\ P_n,\ (\text{and }P_1\text{ if }r=1)\)\\
\addlinespace[2pt]
\(E_6\)
    &\(\mathcal O_{A_1}\)
    &\(\text{none}\)
    &\(P_1,\ P_6\)\\
\addlinespace[2pt]
\(E_7\)
    &\(\mathcal O_{A_1},\ \mathcal O_{2A_1}\)
    &\(\text{none}\)
    &\(P_7\)\\
\bottomrule
\end{tabular*}
\endgroup
\end{center}
For every listed parabolic \(P\), one may take
\[
    X=\mathcal O\cap\mathfrak n_P.
\]
\end{theorem}

\begin{proof}
By \cite[Theorem~1.1]{JiaCotangent}, the table lists all pairs in the
classical types. Fu and Liu constructed the displayed models
for every classical pair
\cite[Table~1 and Theorem~4.2]{FuLiu}. Dynkin-diagram automorphisms give
the remaining listed parabolics. No pair occurs in types \(A_1\) or
\(A_2\). By \cite[Theorem~5.1]{JiaCotangent}, no nonzero orbit closure in
types \(G_2\), \(F_4\), or \(E_8\) admits a cotangent model.

By \cite[Theorem~5.2]{JiaCotangent}, the only exceptional nilpotent
orbits not excluded by the nonexistence results are
\[
    \mathcal O_{A_1}\subset\mathfrak e_6,
    \qquad
    \mathcal O_{A_1},\ \mathcal O_{2A_1}\subset\mathfrak e_7.
\]
Fu and Liu constructed the stated cotangent models for
\((\mathfrak e_6,\mathcal O_{A_1})\) and \((\mathfrak e_7,\mathcal O_{A_1})\)
\cite[Table~1 and Theorem~4.2]{FuLiu}. For each of these two pairs and
every corresponding parabolic listed above, their proof showed that
\(X=\mathcal O\cap\mathfrak n_P\) is smooth and quasi-affine and that
\[
    (T^*X)^{\mathrm{aff}}\cong\overline{\mathcal O}.
\]

For the remaining case \(\mathcal O_{2A_1}\subset\mathfrak e_7\), set
\(P=P_7\). Then \(\mathfrak n_P=\mathfrak u^+\). By
Proposition~\ref{intersection}, the variety
\[
    X=\mathcal O_{2A_1}\cap\mathfrak n_P
\]
is smooth and quasi-affine. By Theorem~\ref{main}, we have
\[
    (T^*X)^{\mathrm{aff}}\cong\overline{\mathcal O}_{2A_1}.
\]
By \cite[Theorems~1.1, 5.1, and~5.2]{JiaCotangent}, no other pair admits
a cotangent model. This proves the theorem.
\end{proof}

\bigskip
\noindent
Boming Jia

\noindent
Yau Mathematical Sciences Center,\\
Jingzhai 301, Tsinghua University,\\
Beijing 100084, China

\noindent
Email: \href{mailto:jiabm@tsinghua.edu.cn}{jiabm@tsinghua.edu.cn}

\end{document}